\documentclass[a4paper,12pt,reqno]{amsart}

\usepackage{amssymb,amsfonts,amscd,latexsym}
\usepackage{amsthm}
\usepackage{enumerate}
\usepackage{amsmath, mathrsfs,bm}
\usepackage{graphics,color}
\usepackage{xy}
\xyoption{matrix} \xyoption{arrow}

\numberwithin{equation}{section}

\begingroup
\newtheorem{thm}{Theorem}[section]
\newtheorem{cor}[thm]{Corollary}
\newtheorem{prop}[thm]{Proposition}
\newtheorem{lem}[thm]{Lemma}

\theoremstyle{definition}

\newtheorem{rmk}[thm]{Remark}

\newtheorem*{acknowledgement}{Acknowledgement}
\endgroup

\newcommand{\rada}{{\textup{rad}\kern.3pt A}}
\newcommand{\El}[1]{{\mathcal{E}\kern-2.6pt\ell({#1})}}
\newcommand{\Ell}{{\mathcal{E}\kern-2.6pt\ell}}
\newcommand{\ElAX}{{\mathcal{E}\kern-2.6pt\ell(A,L(X))}}
\newcommand{\ElnAX}{{\mathcal{E}\kern-2.6pt\ell_n(A,L(X))}}
\newcommand{\ElthAX}{{\mathcal{E}\kern-2.6pt\ell_3(A,L(X))}}

\newcommand {\N}{\mathbb{N}} 
\newcommand {\CC}{\mathbb{C}} 
\newcommand{\Dim}{{\sf dim}\,}
\newcommand{\lDim}{{\sf ldim}\,}

\newcommand{\spa}{{\mkern.5mu\rm span}\mkern1.5mu}

\newcommand{\Tr}{{\sf Tr}}
\newcommand{\zset}[1]{{\boldsymbol{#1}\zeta}}

\def\C*{{\sl C*}-algebra}
\def\Cs*{{\sl C*}-subalgebra}

\newcommand {\be}{\begin{equation}}
\newcommand {\ee}{\end{equation}}
\newcommand {\beq}{\begin{eqnarray*}}
\newcommand {\eeq}{\end{eqnarray*}}

\begin{document}
\title[Locally quasi-nilpotent elementary operators]{Locally quasi-nilpotent elementary operators}
\author{Nadia Boudi}
\address{D\' epartement de Math\' ematiques, Universit\' e Moulay Ismail, Facult\'e des Sciences, Mekn\`es, Maroc}
\email{nadia\_boudi@hotmail.com}
\author{Martin Mathieu}
\address{Pure Mathematics Research Centre, Queen's University Belfast, University Road, Bel\-fast BT7 1NN, Northern Ireland}
\email{m.m@qub.ac.uk}

\begin{abstract}
Let $A$ be a unital dense algebra of linear mappings on a complex vector space~$X$.
Let $\phi=\sum_{i=1}^n M_{a_i,b_i}$ be a locally quasi-nilpotent elementary operator of length~$n$ on~$A$.
We show that, if $\{a_1,\ldots,a_n\}$ is locally linearly independent, then the local dimension of
$V(\phi)=\spa\{b_ia_j: 1 \leq i,j \leq n\}$ is at most $\frac{n(n-1)}{2}$.   If  $\lDim V(\phi)=\frac{n(n-1)}{2} $, then there
exists a  representation of $\phi$ as $\phi=\sum_{i=1}^n M_{u_i,v_i}$ with $v_iu_j=0$ for $i\geq j$.
Moreover, we give a complete characterization of locally quasi-nilpotent elementary operators of length~$3$.
\end{abstract}


\maketitle

\section{Introduction}\label{sect:intro}

\noindent
Let $A$ be a unital complex algebra.
In order to gain information on the range of a linear mapping $\phi\colon A\to A$ one can impose
conditions on the size of the spectrum of the elements $\phi(x)$, $x\in A$. Various classes of linear mappings
such as (inner) derivations or generalised (inner) derivations with small spectrum in the range,
especially when the spectrum of every element in the range consists only of zero, have been studied
by Aupetit, Bre\v sar, Le Page, Pt\'ak, \v Semrl, the present authors and many others; see, e.g.,
\cite{Aup, ChKeLee, NaMa04, NaMa11, Pta} and the references contained therein.
Quite often such results are connected with commutativity criteria; see, e.g., \cite[\S2 in Chapter~V]{Aup}.

An attractive and fairly large class of linear mappings are the elementary operators, that is, those which
can be written as $\phi(x)=\sum_{i=1}^n a_ixb_i$, $x\in A$. In this paper we intend, on the one hand,
to extend various results that were available only for special elementary operators, for instance, for generalised
inner derivations $x\mapsto ax-xb$, to general elementary operators; on the other hand, we want to uncover
the algebraic structure behind arguments which were at times used in an analytic setting.

Clearly, every elementary operator leaves each ideal of $A$ invariant and thus induces an elementary operator
on each primitive quotient of~$A$. As a result, we focus our attention here solely on the case of dense algebras of linear mappings
on a complex vector space. The general case then is mainly a question
of putting the information from primitive quotients together to a global picture which can, e.g., be done with the aid
of the extended centroid of a semisimple algebra~$A$.

In the setting of a dense algebra $A$ of linear mappings on a vector space~$X$,
it turns out that the local dimension of $\spa \{a_1, \ldots, a_n\}$  plays an important role in such descriptions.
In the case where $\{a_1, \ldots, a_n\}$ is locally linearly independent, we will determine the maximal local dimension
of $\spa \{b_i a_j: 1 \leq i,j \leq n\}$, and we shall give a complete characterization if the local dimension is maximal.
Our arguments depend heavily on Gerstenhaber's well-known theorem concerning linear spaces of nilpotent matrices~\cite{Ger}.
Lacking a general description of maximal nilpotent spaces, and
since the classification of minimal locally linearly dependent spaces is a highly non-trivial matter \cite{BrSe, MeSe},
we will, in the last part of the paper (Section~\ref{sect:length-three}), have to limit ourselves to elementary operators of length at most three.

The main idea in our approach is as follows.
Starting with a locally quasi-nilpotent elementary operator~$\phi$ (that is, the spectrum of $\phi(x)$ is the singleton~$0$
for every $x\in A$), we aim to find finite-dimensional $\phi(x)$-invariant subspaces $L\subseteq X$, for certain~$x$, so that
$\phi(x)_{|L}$ will be nilpotent. If the dimension of $L$ is large relative to the local dimension of $\spa \{a_1, \ldots, a_n\}$,
we will get full information on how to represent $\phi$ with suitable coefficients. The details of this approach will
be explained at the start of Sections~\ref{sect:locally-nil} and~\ref{sect:length-three}.

In a sequel to this paper,~\cite{NaMa14}, we shall apply our results to the study of spectrally bounded elementary operators
on Banach algebras continuing the line of research in~\cite{CuMa} and~\cite{NaMa11}, for example.

\section{Preliminaries}\label{sect:pres}

\noindent
Let $X,Y$ be two complex vector spaces. The space of all linear mappings from  $X$ to $Y$ will be designated by~$L(X, Y)$, and
$L(X)$ stands for the algebra $L(X,X)$.
Throughout this paper, $A$ will denote a dense algebra (in the sense of Jacobson) of linear mappings on~$X$.

Though not quite standard, we will use the term \textit{elementary operator\/} for every linear mapping
$\phi\colon A\to L(X)$ that can be written in the form
\begin{equation}\label{eq:elop-def}
\phi(x)=\sum_{i=1}^n a_ixb_i\quad(x\in A),
\end{equation}
for some $a_i,\,b_i\in L(X)$ and some $n\in\N$. Special cases are $L_a\colon x\mapsto ax$, $R_b\colon x\mapsto xb$ and
$M_{a,b}=L_aR_b$. Clearly the representation  of $\phi$ in a sum as in~\eqref{eq:elop-def} is not unique; but
it is well known that, in our setting, whenever
we change the coefficients $a_i$, $b_i$ to new ones $u_i$, $v_i$ (and possibly the~$n$ to some~$m$)
we obtain $u_i\in\spa\{a_1,\ldots,a_n\}$ and $v_i\in\spa\{b_1,\ldots,b_n\}$.
The smallest $n\in\N$ such that the elementary operator $\phi$ can be written as $\phi=\sum_{i=1}^n M_{a_i,b_i}$
is called \textit{the length of}~$\phi\,$ and will be abbreviated as~$\ell(\phi)$. If $\phi=\sum_{i=1}^n M_{a_i,b_i}$
and $\ell(\phi)=n$ then, evidently, the sets $\{a_1,\ldots,a_n\}$ and $\{b_1,\ldots,b_n\}$ are linearly independent.
Moreover, if $\phi=\sum_{i=1}^n M_{u_i,v_i}$, there exists an invertible matrix $P\in M_n (\CC)$ such that
\[
(v_i u_j)_{1 \leq i,j \leq n}= P^{-1} (b_i a_j)_{1 \leq i,j \leq n}\, P.
\]

We will denote by $\ElAX$ and $\ElnAX$, respectively the set of elementary operators from $A$ to $L(X)$ of arbitrary length and
of fixed length~$n$.
If $ \phi= \sum_{i=1}^n {M_{a_i,b_i}} \in \ElAX$, we define a new elementary operator $\phi^*$ by  $\phi^*= \sum_{i=1}^n M_{b_i,a_i}$.

We shall use the following notation for  $\phi =\sum_{i=1}^n {M_{a_i,b_i}}$:
\begin{equation*}\label{eq:notation}
\begin{split}
L (\phi)  &= \spa \{a_1, \ldots, a_n \},\\
R (\phi)  &= \spa \{b_1, \ldots, b_n\},\text{and}\\
V(\phi)   &= \spa \{b_i a_j: 1 \leq i,j \leq n\}.
\end{split}
\end{equation*}

A linear mapping $y\in L(X)$ is said to be \textit{quasi-nilpotent\/} if $\lambda-y$ is bijective for each $\lambda\in\CC\setminus\{0\}$
and $y$ itself is not bijective. We shall call $\phi\in\ElAX$ \textit{locally quasi-nilpotent\/} if $\phi(x)$ is quasi-nilpotent
for every $x\in A$.

\smallbreak
For a finite-dimensional subspace $V$ of~$L(X,Y)$,
let $\lDim V= \max \{ \Dim V \zeta: \zeta \in X \}$ denote the \textit{local dimension of}~$V$, cf.~\cite{LiPa}.
Recall that $V$ (or, equivalently, a basis of~$V$) is said to be \textit{locally linearly dependent\/}  if $\lDim V < \Dim V$.
For background on these notions, the reader is referred to \cite{BrSe, MeSe}.
In the case that $\lDim V= \Dim V$, any vector satisfying $\Dim V \zeta=\Dim V$ is called a \textit{separating vector of}~$V$.

Spaces of matrices with bounded rank have been extensively studied by many authors, see, e.g.,~\cite{AtLl}, and their classification
is still a challenging problem. The connection with  locally linearly dependent spaces is given by the following lemma the proof of which
is straightforward and hence omitted.

\begin{lem}\label{mbr}
Let $X,Y$ be complex vector spaces, and let $V$ be a finite-dimensional subspace of $L(X,Y)$.
Suppose that $\lDim V=r$ and set $\widehat{X}= \{ \widehat{\zeta}: \zeta \in X\}$, where $\widehat{\zeta}:  V \rightarrow Y$ is
defined by $\widehat\zeta (T)= T\zeta$ for all $T \in V$. Then $\widehat{X}$ is a subspace of $L(V,Y)$ of rank at most~$r$.
\end{lem}

Let $F(X)$ denote the ideal of $L(X)$ consisting of those linear mappings with finite-dimensional range
and let $X^*$ be the space of all linear functionals on~$X$.
For a vector $\xi\in X$ and a linear functional $f\in X^*$, let $\xi\otimes f\in F(X)$ denote the mapping
$(\xi\otimes f)(\eta)=f(\eta)\,\xi$, $\eta\in X$.

\smallskip
The first of the next two auxiliary results is standard while the second follows from \cite[Lemma 2.1]{BrSe}; see also~\cite{GoLaWo}.
\begin{lem}\label{form}
Let $\phi=\sum_{i=1}^m {M_{a_i,b_i}}\in\ElAX$. Suppose that $\{u_1, \ldots, u_n\}$ is a basis of $L(\phi)$.
Then there exists a basis $\{v_1, \ldots, v_n\}$ of $R(\phi)$ such that $\phi= \sum_{i=1}^n {M_{u_i,v_i}}$.
\end{lem}
\begin{lem}\label{free}
Let $X$ be a vector space and let $V_1, \ldots, V_k$ be finite-dimensional subspaces of~$L (X)$.
Then there exists $\zeta \in X$ such that $\Dim V_i \zeta= \lDim V_i$ for all $1 \leq i \leq k$.
\end{lem}
\begin{prop}\label{prop:phi-nilpotent}
Let $\phi=\sum_{i=1}^n \phi_i$, where $\phi_i\colon A \rightarrow L (X)$  are linear mappings satisfying $\phi_j (x) \phi_i (y) =0$ for all
$j \geq i$ and $x, y \in A$. Then $\prod_{t=1}^{n+1} \phi (x_t)=0$ for all $x_1, \ldots, x_{n+1} \in A$.
\end{prop}
\begin{proof}
We proceed by induction on~$n$.   The case $n=1$ is trivial.
Assume that $n>1$ and the desired conclusion holds for $n-1$.
Since $\phi_n (x)\phi_i(y)=0$ for all~$i$ and all $x,y \in A$, we have for each~$n$,
\begin{equation*}
\phi (x )\phi(y)= \sum_{i=1}^{n-1} \phi_i (x)\, \sum_{i=1}^{n-1} \phi_i (y) + \sum_{i=1}^{n-1} \phi_i (x) \phi_n (y)
\end{equation*}
 and analogously for all $r$ and $x_1, \ldots, x_r \in A$,
 \begin{equation*}
 \prod_{t=1}^r \phi(x_t)= \prod_{t=1}^r \sum_{i=1}^{n-1} \phi_i (x_t)  + \prod_{t=1}^{r-1}\sum_{i=1}^{n-1} \phi_i (x_t) \phi_n(x_r).
 \end{equation*}
 The induction hypothesis implies that $\prod_{t=1}^n\sum_{i=1}^{n-1} \phi_i (x_t)=0$. Thus
 \begin{equation*}
\prod_{t=1}^n \phi (x_t)= \prod_{t=1}^{n-1}\sum_{i=1}^{n-1} \phi_i (x_t) \phi_n(x_n).
 \end{equation*}
 Consequently, $\prod_{t=1}^{n+1} \phi (x_t)=0$.
\end{proof}

The following consequence of the above proposition explains a local nilpotency property of elementary operators and
is in part the motivation for our main result in Section~\ref{sect:length-three} of this paper.

\begin{cor}\label{prop:elop-nilpotent}
Let $\phi=\sum_{i=1}^n {M_{a_i,b_i}}\in\ElAX$. Suppose that $b_ia_j=0$ for all $i \geq j$. Then $(\phi(x))^{n+1}=0$ for all $x\in A$.
\end{cor}

\section{Locally nilpotent elementary operators}\label{sect:locally-nil}

\noindent
The purpose of this section is to find necessary conditions for local nilpotency of an elementary operator.

Let  $\phi=\sum_{i=1}^n  {M_{a_i,b_i}}$ be an elementary operator on $A$ of length~$n$, where  $a_i, b_i \in L(X)$.
Let $\zeta\in X$ and $x \in A$ be such that $xV(\phi)\zeta\subseteq \mathbb C \zeta$.  Then
\be \label{01}
\phi (x)L(\phi)\zeta\subseteq L(\phi)\zeta.
\ee
Let $\pi\colon\CC \zeta \rightarrow \CC$, $\pi(\zeta)=1$ denote the canonical map.
Suppose moreover $\{a_1 \zeta, \ldots, a_n \zeta\}$ is linearly independent; in such a situation, we shall abbreviate
this set to $\zset a$, if convenient.
Let  $ M( \phi (x), \zset a)$ be the corresponding matrix representation of $\phi (x)$ with respect  to $\zset a$. Then
\be\label{02}
M( \phi (x), \zset a)= (\pi(x b_i a_j \zeta))_{1 \leq i, j \leq n}.
\ee
\begin{prop}\label{trace-qu}
Let $\phi=\sum_{i=1}^n {M_{a_i,b_i}}\in\ElAX$ be an elementary operator.
Suppose that for every $\zeta \in  X$ and $x \in A$ such that $xV(\phi)\zeta \subseteq \mathbb C \zeta$,
$\phi (x)_{|L(\phi) \zeta}$ is nilpotent.
Then $\sum_{i=1}^n b_i a_i =0$.
\end{prop}
\begin{proof}
Suppose that there exist $\zeta\in X$   and $j\in \{1, \ldots, n\}$ such that  $b_j a_j\zeta \neq 0$.  Set $s= \Dim L(\phi)\zeta$.
Rearranging the order of the $a_i$'s, if necessary, we may assume that $\{a_1 \zeta, \ldots, a_s \zeta\}$ is linearly independent.
For $t>s$ we can thus write $a_t\zeta=\sum_{i=1}^s\lambda_{ti}a_i\zeta$ with unique $\lambda_{ti}\in\CC$, $s<t\leq n$.
Replacing the coefficients $a_t$, $t>s$ by
\[
a_t'=a_t-\sum_{i=1}^s\lambda_{ti}a_i,\quad s<t\leq n
\]
it is clear that $a_t \zeta=0$ for $t>s$. As it is easily verified that
\be
\sum_{i=1}^n {M_{a_i,b_i}}=\sum_{i=1}^s  M_{a_i, b_i+ \sum\limits_{t=s+1}^n \lambda_{ti} b_t}+ \sum_{t=s+1}^n  {M_{a_t',b_t^{}}} ,
\ee
we can write $\phi$ in this new representation assuming these two additional properties of the coefficients~$a_i$.
(This argument will be used repeatedly in the sequel.)

Choose $i_1, \ldots, i_r$ such that $\{b_{i_1} a_{i_1} \zeta, \ldots, b_{i_r} a_{i_r} \zeta\}$ is linearly independent, $r$ being maximal.
Fix $j \in \{1, \ldots, r\}$. Choose $x \in A$ such that
\be
xV(\phi)\zeta\subseteq\CC\zeta, \quad x(b_{i_j} a_{i_j}\zeta)=\zeta\quad\text{and}\quad x(b_{i_t} a_{i_t}\zeta)=0\quad(1\leq t\leq r,\; t\neq j).
\ee
Let $J= \{1, \ldots, s\}\setminus\{i_1, \ldots, i_r\}$. Write $b_i a_i \zeta=\sum_{t=1}^r \alpha_t^i b_{i_t} a_{i_t} \zeta$ for all $i \in J$.
Then $x b_{i} a_{i}\zeta= \alpha_j^i  \zeta$ for all $i\in J$.  Moreover,
\begin{equation*}
\phi (x)L( \phi) \zeta\subseteq L(\phi) \zeta \quad\text{and} \quad \Tr (\phi(x)_{|_{L(\phi) \zeta}})= 1+ \sum_{i \in J}\alpha_j^i,
\end{equation*}
where $\Tr$ denotes the trace.
By our hypothesis, we must have $1+ \sum_{i \in J}\alpha_j^i =0$. From
\begin{equation*}
\begin{split}
\sum_{i=1}^n b_i a_i \zeta= \sum_{i=1}^s b_i a_i \zeta &= \sum_{t=1}^r b_{i_t} a_{i_t} \zeta + \sum_{i \in J} b_i a_i \zeta\\
   &= \sum_{t=1}^r b_{i_t} a_{i_t} \zeta+ \sum_{i \in J} \sum_{t=1}^r \alpha_t^i b_{i_t} a_{i_t} \zeta
\end{split}
\end{equation*}
we obtain
\be
\sum_{i=1}^n b_i a_i \zeta=\sum_{t=1}^r \bigl(1+ \sum_{i \in J} \alpha_t^i\bigr) b_{i_t} a_{i_t} \zeta.
\ee
As a result, $\sum_{i=1}^n b_i a_i \zeta=0 $.

Since this entails that $\sum_{i=1}^n b_i a_i \zeta=0$ for every $\zeta\in X$, $\sum_{i=1}^n b_i a_i =0$ as desired.
\end{proof}
\begin{prop}\label{dimV=1}
Let $\phi\in\ElAX$ be an elementary operator such that $\Dim V(\phi)=1$.
Set  $\lDim L(\phi)=r$ and suppose that for every $\zeta \in  X$ and $x \in A$ such that $xV(\phi)\zeta\subseteq \mathbb C \zeta$,
$\phi (x)_{|L(\phi) \zeta}$ is nilpotent. Then $\phi$ admits  a representation of the form $\phi= \sum_{i=1}^n {M_{u_i,v_i}}$ with
 \be
 (v_iu_j)_{1\leq i, j \leq n}=  \left(
                                  \begin{array}{cc}
                                    T & 0 \\
                                    * & 0 \\
                                  \end{array}
                                \right),
 \ee
where $T$ is a strictly upper triangular matrix of order~$r$. Moreover, $\phi^* \phi=0$ and $(\phi (x))^{r+2}=0$ for all $x\in A$.
\end{prop}
\begin{proof}
Write $\phi= \sum_{i=1}^n  {M_{a_i,b_i}}$. It follows from Lemma~\ref{free} that there exists $\zeta \in X$ such that $V(\phi)\zeta\neq0$
and $\Dim L(\phi) \zeta =r$. As in the proof of Proposition~\ref{trace-qu}, we can suppose that $\{a_1 \zeta, \ldots, a_r \zeta\}$ is linearly independent
and that $a_k \zeta=0$ for every $k >r$. It then follows from our assumption on $V(\phi)$ that  $R (\phi) a_k =0$ for all $k >r$.
Let $s,t \in \{ 1, \ldots, n \}$ be such that $b_s a_t \neq 0$. Then $V(\phi)= \CC\, b_s a_t $. Pick $x \in A$ with the property that $x b_s a_t \zeta= \zeta$.
Then $\phi (x) \bigl(\spa \{a_1 \zeta, \ldots, a_r \zeta\}\bigr) \subseteq \spa \{a_1 \zeta, \ldots, a_r \zeta\}$ and we have
\be
M( \phi (x), \zset a)= (\pi(x b_i a_j \zeta))_{1 \leq i, j \leq r}.
\ee
Choose a basis $\{u_1 \zeta, \ldots, u_r \zeta\}$ of $L(\phi) \zeta$ such that $M( \phi(x), \zset u)$ is upper triangular.
Write $u_i \zeta = \sum_{t=1}^r \alpha^i_t a_t \zeta$.
Replacing $u_i$ by $u_i=\sum_{t=1}^r \alpha^i_t a_t$, if necessary, we have
$\spa \{u_1, \ldots, u_r\} = \spa\{a_1, \ldots, a_r\}$. For $k >r$, set $u_k= a_k$ and write $\phi=\sum_{i=1}^n {M_{u_i,v_i}}$
for some suitable $v_1, \ldots, v_n \in L(X)$.
It is clear that the matrix $(v_iu_j)$ has the desired form. Next observe that $\phi$ is the sum of two elementary operators
$\phi_1$, $\phi_2$, such that  $\ell( \phi_1)=r$,  $\ell(\phi_1)+ \ell(\phi_2)=\ell(\phi)$  and $R(\phi) L(\phi_2)=0$. Therefore
\be
(\phi(x))^k= (\phi_1(x))^k+ \phi_2(x) (\phi_1 (x))^{k-1}\quad (k \in \N, \; x \in A).
\ee
It is now easy to get the final assertion.
\end{proof}

The main result of this section gives a description of locally nilpotent elementary operators
where the coefficient spaces have maximal local dimension.

\begin{thm}\label{Ger-qu}
Let $\phi\in\ElnAX$ be an elementary operator of length~$n$ such that $\lDim L (\phi)=n$.
Suppose that for every $\zeta \in  X$ and $x \in A$ such that $xV(\phi)\zeta\subseteq \mathbb C \zeta$,  $\phi (x)_{|L(\phi) \zeta}$ is nilpotent.
Then $\lDim V (\phi) \leq \frac{n(n-1)}{2}$. Moreover, if $ \lDim V (\phi)= \frac{n(n-1)}{2}$, then $\phi$ admits  a representation of the form
$\phi= \sum_{i=1}^n {M_{u_i,v_i}}$, where $v_i u_j =0$  for every $i\geq j$. In particular, $\Dim V (\phi)= \frac{n(n-1)}{2}$.
\end{thm}
\begin{proof}
Write $\phi= \sum_{i=1}^n  {M_{a_i,b_i}}$, set $\lDim V(\phi)= r$ and  choose $\zeta \in X$ such that
$\Dim V(\phi)\zeta= r$ and $\Dim L(\phi)\, \zeta =n$
(Lemma~\ref{free}). Let $\{ b_{i_t} a_{j_t} \zeta \}_{1 \leq t \leq r }$ be a basis of $V(\phi)\zeta$.
Pick $x_1, \ldots, x_r \in A$ with $x_k b_{i_t} a_{j_t} \zeta = \delta_{kt} \zeta$, for each $1 \leq t, k \leq r$.
Let $N$ be the vector subspace of $ M_n(\mathbb C)$ generated by $M( \phi (x_t), \zset a)$.
Then $N$ is nilpotent and $\Dim N=r$. Applying Gerstenhaber's theorem on maximal spaces of nilpotent matrices~\cite{Ger},
we infer that $r \leq \frac{n(n-1)}{2}$ and if $r=  \frac{n(n-1)}{2}$,
there exists a basis $B$ of $L(\phi) \zeta$ such that $M(\phi (x), B)$ is upper triangular for each $x \in \spa \{ x_1, \ldots, x_r\}$.
Next suppose that $r=  \frac{n(n-1)}{2}$  and set $B= \{u_1 \zeta, \ldots, u_n \zeta  \}$. Then $L (\phi)=\spa \{u_1, \ldots u_n  \} $.
Applying Lemma~\ref{form}, we get the existence of a basis $\{v_1, \ldots, v_n  \}$ of $R (\phi)$ such that $\phi= \sum_{i=1}^n {M_{u_i,v_i}}$.
Since $M( \phi (x), B)= (\pi(x v_i u_j \zeta))_{1 \leq i,j \leq n}$, we have $\pi(x v_i u_j \zeta)=0$ for all $i \geq j$ and $x \in \spa \{x_1, \ldots, x_r \}$,
thus  $v_i u_j \zeta =0$ for all $i \geq j$. Since $\Dim V(\phi)\zeta= \frac{n(n-1)}{2}$, the set $ \{ v_i u_j \zeta: i <j \} $ is a basis of $V(\phi)\zeta$.

Next we shall show by induction on $n$ that $\Dim V=\frac{n(n-1)}{2}$.  The case $n=1$ is trivial.
Suppose that the desired result is true for elementary operators of length $n-1$, $n>1$.
We first claim that $R(\phi) u_1 = 0$. To prove the claim suppose that there exists $\zeta' \in X$ such that $R(\phi) u_1 \zeta' \neq 0$.
Let $i_1$ be maximal (among $1, \ldots, n $) such that $v_{i_1}u_1 \zeta' \neq 0$. It  is clear that  $i_1 >1$.
Let $i_2$ be maximal such that $\{ v_{i_1} u_1 \zeta', v_{i_2} u_1 \zeta' \}$ is linearly independent.
If $i_2 < t <i_1$, replace $v_t$ by $v'_t= v_t- \alpha v_{i_1}$
and $u_{i_1}$  by $u'_{i_1}= u_{i_1}+ \alpha u_t$ such that $v'_t u_1 \zeta'=0$.
We recursively find $i_1> \ldots > i_r$ and define $u'_i, v'_i$ such that $u'_1=u_1$,  $\phi= \sum M_{u'_i,v'_i}$,
$v'_j u'_i \zeta=0$ for all $j \geq i$, the sets $\{ v'_i u'_j \zeta: i <j\}$ and
$\{ v'_{i_1} u'_1 \zeta', \ldots, v'_{i_r} u'_1 \zeta' \}$ are linearly independent   and if $t \not\in \{i_1, \ldots, i_r \}$, $v'_t u'_1 \zeta'=0$.
Choose a non-zero scalar $\lambda$ such that $ \{ v'_iu'_j (\zeta+ \lambda \zeta'): i <j \}$ is linearly independent.
Replacing $\zeta'$ by $\zeta+ \lambda \zeta'$ if needed,  we may assume with no loss of generality that $ \{ v'_i u'_j \zeta': i <j \}$ is linearly independent.
Thus, by the above argument,  there exists $c_1 \in L(\phi)$ such that $R(\phi)c_1 \zeta'=0$. In particular, $c_1 \not\in \CC u'_1$.
Write $c_1= \sum _{i=1}^n \beta_i u'_i$. Since $v'_1 c_1 \zeta'=0$ and $\{ v'_1u'_2 \zeta', \ldots, v'_1 u'_n \zeta'\}$ is linearly independent,
we infer that $ \beta_1 v'_1 u'_1 \zeta'\neq 0$. Consequently, $i_r=1$. We get a contradiction by choosing  $x \in A$ such that
$x v'_{i_t}u'_1 \zeta'=0$ for $t=1, \ldots, r-1$ and $x v'_1 u'_1 \zeta'=\zeta'$ (indeed, we have $\phi (x)u'_1 \zeta'= u'_1 \zeta'$).
The claim is proved. Now set $\psi= \sum_{i=2}^n {M_{u_i,v_i}}$. Then for every $x \in A$ and $\zeta' \in X$ such that
$\psi(x) V(\psi) \zeta' \subseteq \CC \zeta'$, the restriction of $\psi (x)$ to $L(\psi) \zeta'$ must be nilpotent; otherwise,
there exists $y \in A$ such that $yV(\phi)\zeta' \subseteq \CC \zeta'$ and $y_{|V(\psi)\zeta'}= x _{|V(\psi)\zeta'}$ such that
$\phi (y)_{|V(\phi)\zeta'}$ is not nilpotent.
Our induction hypothesis yields the desired result.
\end{proof}
\begin{prop}\label{ldimV=1}
Let $\phi= \sum_{i=1}^n  {M_{a_i,b_i}}$ be an elementary operator of length $n$ such that $\lDim  L (\phi)=n$ and $\lDim V(\phi)=1$.
Suppose that for every $\zeta \in  X$ and $x \in A$ such that $xV(\phi)\zeta\subseteq\mathbb C \zeta$,
$\phi (x)_{|L(\phi) \zeta}$ is nilpotent. Then $\Dim V (\phi) \leq \frac{n(n-1)}{2}$. Moreover, if $ \Dim V(\phi)= \frac{n(n-1)}{2}$,
then $\phi$ admits  a representation of the form $\phi= \sum_{i=1}^n{M_{u_i,v_i}}$ where $v_i u_j =0$  for every $i\geq j$.
\end{prop}
\begin{proof}
Choose $i, j$ such that $b_i a_j \neq 0$.
The case $\Dim V(\phi)=1$ follows from Proposition~\ref{dimV=1}, so suppose that $\Dim V(\phi)>1$.
Then there exists $k,l$ such that $b_k a_l \not \in \CC b_i a_j$.
We have $\{b_i a_j \zeta, b_k a_l \zeta\}$ linearly dependent for each $\zeta \in X$.
It follows from \cite[Theorem 2.3]{BrSe} that $b_i a_j$ has rank  one. Thus, there exists $\zeta_0 \in X$ and linear functionals
$f_{kl}$ such that $b_k a_l= \zeta_0 \otimes f_{kl}$ for all $1 \leq k, l \leq n$.  Set
$M(\zeta)= (f_{ij} (\zeta))_{1 \leq i, j \leq n}$ for every $\zeta \in X$. Clearly, $M(\zeta)+ M(\zeta')= M (\zeta+ \zeta')$.
Observe that if $\Dim L(\phi) \zeta=  n$, and $x \in A$ satisfies $x \zeta_0= \zeta$, then $x(V(\phi)) \zeta) \subseteq \CC \zeta$ and
$M( \phi (x), \zset a)= M(\zeta)$. Hence $M(\zeta)$ is nilpotent. Set $\Dim V(\phi)=r$. Then $\Dim \spa\{f_{ij}: 1 \leq i, j \leq n \} =r$.
Choose $\zeta_1, \ldots, \zeta_r$ with the property that the set $\{ M(\zeta_1), \ldots, M(\zeta_r) \}$ is linearly independent and
pick $\zeta \in X$ such that $\Dim L(\phi) \zeta=n$. It follows from \cite[Lemma 2.1]{BrSe} that for all but finitely many $r+1$-tuples
$(\lambda_0, \lambda_1, \ldots, \lambda_r) \in \CC^n$, $\Dim L(\phi) (\lambda_0 \zeta + \sum_{t=1}^r \lambda_t \zeta_t)=n$  and
$\{ M(\zeta_1 + \lambda_0 \zeta), \ldots, M(\zeta_r+ \lambda_0 \zeta) \}$ is linearly independent.
In particular, for all but finitely many $(\lambda_0, \lambda_1, \ldots, \lambda_r) \in \CC^n$, $M( \lambda_0 \zeta + \sum_{t=1}^r \lambda_t \zeta_t)$
is nilpotent. Let $N$ be the vector subspace of $ M_n(\mathbb C)$ generated by
\[
\{M( \lambda_0 \zeta + \sum_{t=1}^r \lambda_t \zeta_t): (\lambda_0, \lambda_1, \ldots, \lambda_r) \in \CC^n \}.
\]
Then it is easy to see that $N$ is nilpotent and $\Dim N \geq r$. Applying once again Gerstenhaber's theorem on maximal spaces of
nilpotent matrices~\cite{Ger}, we infer that $r \leq \frac{n(n-1)}{2}$ and if $r=  \frac{n(n-1)}{2}$,
there exists an invertible matrix $P \in M_n (\CC)$ such that $P^{-1} NP$ equals the subspace of nilpotent upper triangular matrices.
Suppose that $r = \frac{n(n-1)}{2}$ and let $B= \{u_1, \ldots, u_n\}$ be the basis corresponding to $P$ of $L(\phi)$.
Then $\phi = \sum_i {M_{u_i,v_i}}$ for a convenient basis $\{ v_i \}$ of $R (\phi)$. Set $v_j u_i = \zeta_0 \otimes g_{ij}$.
Then $g_{ij}=0$  for $i \geq j$. This completes the proof.
\end{proof}

The first non-trivial case allows a particularly nice description.

\begin{cor}\label{2-qu}
Let $\phi$ be an elementary operator of length~$2$. Suppose that for every $\zeta \in  X$ and $x \in A$ such that
$x (V (\phi)) \zeta \subseteq \mathbb C \zeta$,  $\phi (x)_{_{|L(\phi) \zeta}}$ is nilpotent.
Then $\phi=M_{a,b} +M_{c,d}$, where $ba=dc=bc=0$.
\end{cor}
\begin{proof}
Write $\phi= \sum_{i=1}^2 {M_{a_i,b_i}}$. Suppose first that $\lDim L (\phi)=2$. Then the desired result  follows from Theorem~\ref{Ger-qu}
and Proposition~\ref{ldimV=1}. Now suppose that $\{a_1, a_2\}$ is locally linearly dependent. Then, by \cite[Theorem 2.3]{BrSe},
there exists $\zeta_0 \in X$ and linear functionals $f_1, f_2$ such that $a_i= \zeta_0 \otimes f_i$.
We have
\be
\phi(x) \zeta= \sum_{i=1}^2 f_i (x b_i \zeta) \zeta_0.
\ee
Suppose for a moment that $b_i \zeta_0 \neq 0$ for some $i$. Since $\{f_1, f_2\}$ is linearly independent,
we can find $x \in A$ such that $f_i (xb_i \zeta_0)=1$ and $f_j (xb_j \zeta_0)=0 $ for $j \neq i$,
and we get a contradiction. Thus $b_i a_j=0$ for $i,j =1,2$.
\end{proof}

For each $y\in L(X)$, we define the \textit{point spectrum} $\sigma_p(y)$ by
\[
\sigma_p(y)=\{\lambda\in\CC:\lambda-y\text{ is not injective}\}.
\]
The following result reveals the relationship between the two conditions ``quasi-nilpotent" and ``of finite rank".
\begin{prop}\label{f.spectra}
Let $\phi= \sum_{i=1}^n  {M_{a_i,b_i}}$ be an elementary operator of length $n$, and let $N \in \N$. Suppose that
$\#\sigma_p (\phi (x)) < N$ for all $x \in A$. Then either $b_i a_j$ has finite rank for all $1 \leq i, j \leq n$ or
there exists a finite-codimensional subspace $Y$ of $X$ such that    for every $\zeta \in  Y$ and $x \in A$
satisfying $xV(\phi)\zeta\subseteq \mathbb C \zeta$,  $\phi (x)_{|L(\phi) \zeta}$ is nilpotent.
\end{prop}
\begin{proof}
Suppose that there exists $i, j \in \{1, \ldots, n\}$ such that $b_i a_j\notin F(X)$ and choose $r>0$ maximal with
the property that  $\{b_{i_1} a_{j_1}, \ldots, b_{i_r} a_{j_r}\}$ is linearly independent  modulo~$F(X)$.
Then for each $i, j$, there exist scalars $\alpha_{ij}^{1}, \ldots, \alpha_{ij}^{r}$
and $y_{ij}\in F(X)$ such that $b_ia_j= \sum_{i=1}^r \alpha_{ij}^{t} b_{i_t}a_{j_t}+y_{ij}$. Set $Y= \bigcap\, \ker y_{ij}$.
Then it is easy to see that $Y$ has finite codimension. Suppose that there exists $\zeta \in Y$ and $x \in A$ satisfying
$xV(\phi)\zeta\subseteq \CC \zeta$ and $\phi (x)_{|L(\phi)\zeta}$ is not nilpotent. With no loss of generality, we may suppose
that there exists $u \in L(\phi)$ such that $(\phi(x)) u\zeta= u \zeta$. Set $xb_i u \zeta= \lambda_i \zeta$ and $a= \sum \lambda_i a_i$.
Then $u\zeta= a\zeta$. Applying \cite[Lemma 2.1]{Nad}, we can find $\zeta_1, \ldots, \zeta_N \in Y$ such that
$ \{b_{i_1} a_{j_1} \zeta_t, \ldots, b_{i_r} a_{j_r}\zeta_t\}_{1 \leq t \leq N}$ is linearly independent.
Let $\mu_1, \ldots, \mu_N$ be distinct numbers. Choose $y \in A$ such that
$y b_{i_k} a_{j_k} \zeta_t = \mu_t \pi(x b_{i_k} a_{j_k} \zeta) \zeta_t$.
Since $\zeta, \zeta_1, \ldots, \zeta_N$ lie in $Y$, for all $i,j$, $y b_i a_j \zeta_t =\mu_t (\pi x b_i a_j \zeta) \zeta_t$.
For each $c \in L(\phi)$, $y b_ic \zeta_t= \mu_t (\pi x b_i c \zeta) \zeta_t$. Now  $\phi (y) a \zeta_t= \sum a_iy b_ia \zeta_t=\mu_t a \zeta_t$.
In particular, $\{\mu_1, \ldots, \mu_N\} \subseteq \sigma_p (\phi (y))$, a contradiction.
\end{proof}

\section{Locally quasi-nilpotent elementary operators of length~$3$}\label{sect:length-three}

\noindent
In this section we shall pay special attention to the case of locally quasi-nilpotent elementary operators of length~$3$;
shorter lengths have been treated elsewhere, see, e.g.,~\cite{NaMa11}.

Suppose that $\phi(x)$ is quasi-nilpotent for every $x\in A$.
For every $x\in A$ and $\zeta\in X$ such that $xV(\phi)\zeta\subseteq\CC\zeta$, we have $\phi(x)L(\phi)\zeta\subseteq L(\phi)\zeta$.
As $\Dim L(\phi)\zeta$ is finite, $\phi(x)_{|L(\phi)\zeta}$ is nilpotent and we can apply the results of the previous section.
Assuming that $\ell(\phi)=n$, the best case occurs when $\lDim L(\phi)=n$, that is, when we have a separating vector for~$L(\phi)$.
If, in addition, $\lDim V(\phi)$ is maximal, then we obtain the best possible representation of~$\phi$; this was achieved in
Theorem~\ref{Ger-qu}. In the case $n=2$, $\lDim V(\phi)$ is at most~$1$
and therefore no complications can occur (Corollary~\ref{2-qu}). However, already in the case $n=3$, the gap between
$\lDim V(\phi)$ and $\lDim L(\phi)$ widens which results in further ``exceptional'' descriptions of $\phi$ as in the theorem
below. Moreover, local nilpotency is no longer sufficient, as is illustrated by an example in Remark~\ref{rem:2-vs-3}.

The proof of our main result is organised according to the three levels of difficulty:
$\lDim L(\phi)=3$, $\lDim L(\phi)=2$ and finally $\lDim L(\phi)=1$.

\begin{thm}\label{3-qu}
Let  $X$ be a complex vector space and let $A$ be a dense algebra of linear operators on~$X$.
Then  $\phi\in\ElthAX$ is a locally quasi-nilpotent elementary operator if and only if
there exists a representation of $\phi$ of the form  $\phi= \sum_{i=1}^3 {M_{u_i,v_i}}$, $u_i, v_i \in L(X)$
such that one of the following three cases occurs:

\smallskip
\begin{enumerate}
\item[{\rm (i)}] $v_i u_j=0$ for all $i \geq j$;
\smallskip
\item[{\rm (ii)}] $(v_iu_j)_{1 \leq i,j \leq 3}= \left(\begin{array}{ccc}
                                             0 & \zeta_1 \otimes f & 0 \\
                                             \zeta_0 \otimes f & 0 & \zeta_1 \otimes f \\
                                             0 & - \zeta_0 \otimes f & 0
                                           \end{array}\right)$,\\
                                           where  $ \{\zeta_0, \zeta_1\} \subseteq X $ is linearly independent  and  $f \in X^*$;
\smallskip
\item[{\rm (iii)}]  $(v_iu_j)_{1 \leq i,j \leq 3}= \left(\begin{array}{ccc}
                                             0 & \zeta_0 \otimes g & 0 \\
                                             \zeta_0 \otimes f & 0 & \zeta_0 \otimes g \\
                                             0 & - \zeta_0 \otimes f & 0
                                           \end{array}\right)$,\\
                                           where   $\zeta_0 \in X$  and  $\{f,g\} \subseteq X^*$ is linearly independent.
\end{enumerate}
\end{thm}
\begin{proof}
Set $\phi= \sum_{i=1}^3 {M_{a_i,b_i}}$, $a_i, b_i \in L(X)$.
Suppose first that  $ L(\phi)$ has a separating vector. We distinguish three cases.

\noindent
\textit{Case 1.} $\lDim V(\phi)=3$. This case follows from  Theorem~\ref{Ger-qu}.

\noindent
\textit{Case 2.} $\lDim V(\phi)=2$.    Pick  $\zeta \in X$ satisfying $\Dim L(\phi) \zeta=3$ and $\Dim V(\phi)\zeta= 2$ (Lemma~\ref{free}). Set
\be
N= \{ M(\phi (x), \zset a): x \in A,\  xV(\phi)\zeta\subseteq  \CC \zeta  \}.
\ee
Clearly, $N$ is a nilpotent subspace of $M_3 (\CC)$ of dimension~$2$.
Let $N'$ be a maximal subspace containing~$N$.
Then $N'$ is conjugate to one of the following subspaces by \cite[Proposition~3]{Fas} (see also~\cite{FoMa}):
\beq
\left\{\left( \begin{array}{ccc}
  0 & \alpha & \beta \\
  0 & 0 & \gamma \\
  0 & 0 & 0
\end{array}\right): \alpha, \beta, \gamma \in \CC \right \}
 \quad \text {or} \quad \left \{ \left( \begin{array}{ccc}
                                            0 & \beta & 0 \\
                                            \alpha & 0 & \beta \\
                                            0 & -\alpha & 0
                                          \end{array} \right): \alpha, \beta \in \CC \right \}.
\eeq
Thus, we distinguish two subcases.

\noindent
\textit{Subcase 2.1.}
There exists $\zeta \in X$ and a representation of $\phi$ of the form
$\phi= \sum_{i=1}^3 M_{u_i,v_i}$, $u_i\in L(\phi)$, $v_i \in R(\phi)$ such that $\Dim L(\phi) \zeta=3$,  $\Dim V(\phi)\zeta= 2$,
and for every $x \in A$ such that $xV(\phi)\zeta\subseteq\CC\zeta$, the matrix $(\pi (xv_iu_j \zeta))_{1 \leq i, j \leq 3}$ is upper triangular.
Then
\begin{equation*}
(v_iu_j \zeta)_{1 \leq i, j \leq 3}= \left(
                                        \begin{array}{ccc}
                                          0 & v_1u_2 \zeta & v_1u_3 \zeta\\
                                          0 & 0 & v_2u_3 \zeta \\
                                          0 & 0 & 0 \\
                                        \end{array}
                                      \right).
\end{equation*}
Clearly, we may assume that one of the sets $\{v_1u_2 \zeta, v_1u_3 \zeta\}$ or $\{v_1u_3 \zeta, v_2u_3 \zeta\}$  is linearly independent.
Using an argument similar to the one used in the proof of  Theorem~\ref{Ger-qu},
we show that $R(\phi) u_1=0$ or $v_3 L(\phi)=0$ and arrive at the desired conclusion.

\noindent
\textit{Subcase 2.2.}
For every $\zeta\in X$ such that $\Dim L(\phi)\zeta=3$ and
$\Dim V(\phi)\zeta= 2$,  there exists a representation of $\phi$ of the form
$\phi= \sum_{i=1}^3 M_{u_i,v_i}$, $u_i\in L(\phi)$, $v_i\in R(\phi)$  depending on~$\zeta$ such that,
for every $x\in A$ satisfying $xV(\phi)\zeta\subseteq\CC\zeta$, the matrix
$(\pi (xv_iu_j\zeta))_{1 \leq i, j \leq 3}$ has the form
\be\label{3.2}
  \left( \begin{array}{ccc}
   0 & \pi (xv_2 u_3 \zeta)& 0 \\
    \pi (xv_2 u_1 \zeta) & 0 & \pi (xv_2 u_3 \zeta) \\
     0 & -\pi (xv_2 u_1 \zeta) & 0
     \end{array} \right).
\ee
Then we have
\be
(v_iu_j \zeta)_{1 \leq i, j \leq 3}= \left(
                                        \begin{array}{ccc}
                                          0 & v_2 u_3\zeta & 0 \\
                                          v_2 u_1 \zeta & 0 & v_2u_3 \zeta \\
                                          0 & -v_2u_1\zeta & 0 \\
                                        \end{array}
                                      \right).
\ee
Now fix a separating vector  $\zeta $ of $L(\phi)$ such that  $V(\phi)\zeta$ has maximal dimension.
Set $ v_2 u_1 \zeta= \eta_1$ and $ v_2 u_3 \zeta= \eta_2$.
Pick  $\eta_3  \in X \setminus \spa \{ \eta_1, \eta_2\}$. Write $X= \spa \{\eta_1, \eta_2, \eta_3\} \oplus Y$, for some subspace $Y$ of~$X$.
Let  $p_1: X \rightarrow \spa\{\eta_1, \eta_2, \eta_3\} $ and $p_2: X \rightarrow Y$ be the natural projections.
Fix $\gamma_1, \gamma_2, \gamma_3 \in \CC\setminus\{0\}$. Let $\zeta' \in X$.  Choose $x_{ \zeta'} \in X$ satisfying
\begin{equation*}
x_{ \zeta'} \eta_i= \gamma_i \zeta', \quad x_{ \zeta'} p_2v_iu_j\zeta'=0\quad  (1 \leq i,j \leq 3).
\end{equation*}
Set
\[
M(\zeta', \gamma_1, \gamma_2, \gamma_3)= \bigl(\pi x_{\zeta'}v_iu_j \zeta' \bigr)_{1 \leq i,j \leq 3}
\]
and
\[
E (\gamma_1, \gamma_2, \gamma_3)= \{M(\zeta', \gamma_1, \gamma_2, \gamma_3): \zeta' \in X\}.
\]
Observe that  if $\zeta'$ is a separating vector of $L(\phi)$, we have
\begin{equation*}
M(\zeta', \gamma_1, \gamma_2, \gamma_3)= M\bigl(\phi(x_{ \zeta'}), \{u_1 \zeta',u_2\zeta' , u_3 \zeta'\}\bigr)
\end{equation*}
which is nilpotent. Since for every $\zeta' \in X$, $\zeta'+ \lambda \zeta$ is a separating vector of $L(\phi)$ for all
but finitely many $\lambda  \in \CC$,  we infer that $E (\gamma_1, \gamma_2, \gamma_3)$ is  a nilpotent subspace of $M_3 (\CC)$.
Next apply once again \cite[Proposition 3]{Fas}.

Suppose first that  $E(\gamma_1, \gamma_2, \gamma_3) $ has dimension $2$ for an open set $O$ of
values $(\gamma_1, \gamma_2, \gamma_3)$ of~$\CC^3$.
Then, for every $(\gamma_1, \gamma_2, \gamma_3) \in O$, $\zeta_1, \zeta_2, \zeta_3 \in X$ and $w_1, w_2, w_3 \in \{v_iu_j: 1 \leq i, j \leq 3\}$,
we must have $\det ((x_{\zeta_j}w_i \zeta_j)_{1 \leq i,j \leq 3})=0$. Let $(i,j) \not\in \{(1,2), (2,3), (2,1), (3,2)\}$.
Suppose for a moment that $v_i u_j \neq 0$. Then $\{v_iu_j, v_2u_1, v_2u_3\}$ is locally linearly dependent.
Since $v_iu_j \zeta=0$, $v_iu_j X \subseteq \spa \{ \eta_1, \eta_2\}$. Hence we may suppose that  $v_iu_j \zeta_2= \eta_1$ and
either $v_ iu_j \zeta_3= \eta_2$, or $v_iu_j \zeta_3=0$. Set $w_1= v_2 (\gamma_2 u_1- \gamma_1 u_3)$,
$w_2= v_i u_j$, $w_3= v_2u_1$,  and $\zeta_1= \zeta$.   Computing the determinant, and using the fact that
$x_{\zeta}v_2(\gamma_2 u_1- \gamma_1 u_3)\zeta=0$, we see that $v_2u_1 \zeta_2=0$, which is impossible since
we can replace $\zeta_2$ by $\zeta_2+ \lambda \zeta$. Thus $v_i u_j=0$.
Let $\zeta'$ be a separating vector of $L(\phi)$. By considering kernels, we see that we must have
$\gamma_2u_1-\gamma_1 u_3 \in \spa \{u_1', u_3'\}$ where $u_1'$, $u_3'$ now depend on $\zeta'$.
Therefore  $\spa \{u_1, u_3\}= \spa \{u_1', u_3'\}$.
Now it is easy to see that we can  put $u_i= u_i'$ for all~$i$. This implies that we must have $v_2u_1= - v_3u_2$ and $v_1u_2=v_2u_3$.

Now suppose that the set  of $(\gamma_1, \gamma_2, \gamma_3)$  for which  $\Dim E(\gamma_1, \gamma_2, \gamma_3) \neq 2$ is
dense in~$\CC^3$.  Then observe that $E(\gamma_1, \gamma_2, \gamma_3) $ must be triangularizable   for an open subset
$(\gamma_1, \gamma_2,\gamma_3)$ of $\CC^3$. By considering a common eigenvector, we see that we must have
 $\gamma_2u_1-\gamma_1 u_3 \in \spa \{u_1^{\zeta'}, u_3^{\zeta'}\}$ for every separating vector $\zeta'$ of $L(\phi)$.
As above, we deduce that we can put  $u_i= u_i^{\zeta'}$ for all $i$.

Next suppose that there exists $\zeta' \in X$ with $ v_2u_1 \zeta'= \eta_3$.  Choose $x \in A$ with $x \eta_3= \zeta$, $x \eta_2= \zeta'$
and $x \eta_1=0$.   Then $\phi (x)(u_1 \zeta'+u_2 \zeta)=u_1 \zeta'+u_2 \zeta$, a contradiction.
Hence $p_1 v_2u_1 X \subseteq \spa \{ \eta_1, \eta_2\}$. Analogously for $p_1v_2u_3$. We have thereby shown that
$v_iu_j X \subseteq \spa\{ \eta_1, \eta_2\}$ for every $i, j$.  Next choose $\zeta' \in X$ with $v_2u_1 \zeta'=0$
(this is possible since the dimension of $X$ is greater than~$2$). Suppose for a moment that $v_2u_3 \zeta' \neq 0$.
Pick $x \in A$ such that $x \eta_1= \zeta'$ and $x v_2u_3 \zeta'= \zeta$. Then $\phi(x) (u_1 \zeta+ u_2 \zeta')=u_1 \zeta+ u_2 \zeta'$,
a contradiction. Hence $v_2u_1$ and $v_2u_3$ have the same kernel. Since they are at most rank~$2$,
$ \spa\{v_2u_3, v_2u_1\}$ is generated by rank one operators. Thus we can assume that $v_2u_3, v_2u_1$ are rank one.
Since they must have the same kernel, we conclude that we must have the form~(ii).

\smallskip\noindent
\textit{Case 3.}
$\lDim V(\phi)=1$. It follows from Propositions~\ref{dimV=1} and~\ref{ldimV=1} that we only need
to treat  the case where $\Dim V(\phi)=2$. So suppose that $\Dim V(\phi)=2$ and set $b_i a_j= \zeta_0 \otimes f_{ij}$
for all $ 1 \leq i,j \leq 3$. For every $\zeta \in X$, let $M(\zeta)= (f_{ij} (\zeta))_{1 \leq i,j \leq 3}$ and consider the vector
space $N= \{ M(\zeta): \zeta \in X\}$. As above, we show that $N$ is a nilpotent subspace of $M_3 (\CC)$, and hence,
either $N$ is triangularizable, or
 \begin{equation*}
 N=  \left\{   \left(
      \begin{array}{ccc}
        0 & \gamma' & 0 \\
        \gamma & 0 & \gamma' \\
        0 & -\gamma & 0 \\
      \end{array}
    \right)   : \gamma,\gamma' \in \CC  \right\}.
 \end{equation*}
Thus, there exists a representation of $\phi$  of the form  $\phi= \sum_{i=1}^3 {M_{u_i,v_i}}$, $u_i, v_i \in L(X)$
such that either  $v_i u_j=0$ for all $i \geq j$ or
 \begin{equation*}
(v_iu_j)_{1 \leq i,j \leq 3}= \left(\begin{array}{ccc}
                                             0 & \zeta_0 \otimes g & 0 \\
                                             \zeta_0 \otimes f & 0 & \zeta_0 \otimes g \\
                                             0 & - \zeta_0 \otimes f & 0
                                           \end{array}\right); \;\;\; \zeta_0 \in X \text { and } f,g \in X^*.
\end{equation*}

We now move to the case  $\lDim L(\phi)=2$. Let $\zeta \in X$ be such that the vector spaces $L(\phi)\zeta$ and~$V(\phi)\zeta$ have maximal dimensions.
Choose $u \in L(\phi)$ such that $u \neq 0$ and $u \zeta=0$. Write $\phi= \sum_{i=1}^3 M_{u_i,v_i}$ where $u_1=u$.
Set $\psi=  \sum_{i=2}^3 M_{u_i,v_i}$. Then arguing as in the proof of Theorem \ref{Ger-qu}, we see that for every
$x \in A$ satisfying $x V(\psi) \zeta \subseteq \CC \zeta$, $\phi(x)_{|L(\psi) \zeta}$ must be nilpotent.
It follows from Corollary \ref{2-qu} that there exists a representation of $\psi= \sum_{i=2}^3 M_{c_i,d_i}$
such that $d_j c_i \zeta=0$ for $j \geq i$. We have thereby shown that there exists a representation of
$\phi= \sum_{i=1}^3 M_{u_i,v_i}$ such that $v_ju_i \zeta=0$ for $j\geq i$. Next we adapt the argument
used in the proof of Theorem~\ref{Ger-qu} to show that either  $R(\phi) u_1=0$ or $v_3 L(\phi)=0$
(by our assumption on $\zeta$ and $V(\phi)$,  $\lDim V(\phi) \leq 3$).
Treating the three occurring cases $\lDim V(\phi)=1,\,2$ or~$3$
separately, it is easy to complete the argument.

Suppose finally that $\lDim L(\phi)=1$. Then there exists $\zeta_0 \in X$ such that $a_i= \zeta_0 \otimes f_i$.
Therefore,  for each $x \in A$ and $\zeta \in X$, we have
\be
\phi(x) \zeta= \sum_{i=1}^3 f_i (xb_i \zeta) \zeta_0.
\ee
Suppose that $b_i \zeta_0 \neq 0$ for some $i$ and choose $x \in A$ such that  $f_i(xb_i \zeta_0)=1$, $f_j (xb_i \zeta_0)=0$ for
$j \neq i$ (this is possible since the set $ \{ f_1, f_2, f_3\}$ is linearly independent).
Then we get a contradiction. Thus, $b_i a_j=0$ for each $i,j$.

To prove the converse, observe that if $ \phi= \sum M_{u_i,v_i}$ and $v_iu_j=0$ for all $i \geq j$, then it follows from
Corollary~\ref{prop:elop-nilpotent} that $\phi(x)^4=0$ for all $x \in A$. A straightforward computation checking
the other  cases shows that $\phi (x)^5=0$ for all $x\in A$. This completes the proof.
\end{proof}

\begin{cor}
Let $\phi\in\ElthAX$ be a locally quasi-nilpotent elementary operator. Then $\phi^* \phi=0$ and $\phi (x)^5=0$ for all $x\in A$.
\end{cor}

\begin{cor}
Let $\phi\in\ElthAX$ be a locally quasi-nilpotent elementary operator. Then one of the following two cases occurs:
\begin{enumerate}
\item[{\rm (i)}] there exists a representation of $\phi$ of the form $\phi= \sum_{i=1}^3 M_{u_i,v_i}$, where $v_i u_j=0$ for all $i \geq j$;
\smallskip
\item[{\rm (ii)}]  the space $\phi (x) L(\phi) X$ has at most dimension $3$ for all $x \in A$, and hence $\Dim \phi(x)^2 X \leq 3$ for all $x\in A$.
\end{enumerate}
\end{cor}
\begin{proof}
Suppose first that there exists a representation of $\phi$  of the form $\phi= \sum_{i=1}^3 M_{u_i, v_i}$ such that
\begin{equation*}
(v_iu_j)_{1 \leq i,j \leq 3}= \left(\begin{array}{ccc}
                                             0 & \zeta_1 \otimes f & 0 \\
                                             \zeta_0 \otimes f & 0 & \zeta_1 \otimes f \\
                                             0 & - \zeta_0 \otimes f & 0
                                           \end{array}\right) ,
\end{equation*}
where $\zeta_0, \zeta_1 \in X$, $f \in X^*$ and $\{  \zeta_0, \zeta_1\}$ is linearly dependent.
Then
\[
\phi(x)  L(\phi) X \subseteq\spa\{ u_2 x \zeta_0, u_2 x \zeta_1,   u_1x \zeta_1-u_3x \zeta_0   \}\quad(x \in A).
\]
Next suppose that there exists a representation of $\phi$  of the form $\phi= \sum_{i=1}^3 M_{u_i, v_i}$ such that
\begin{equation*}
(v_iu_j)_{1 \leq i,j \leq 3}= \left(\begin{array}{ccc}
                                             0 & \zeta_0 \otimes g & 0 \\
                                             \zeta_0 \otimes f & 0 & \zeta_0 \otimes g \\
                                             0 & - \zeta_0 \otimes f & 0
                                           \end{array}\right),
\end{equation*}
where  $\{f,g\}\subseteq X^*$ is linearly independent and $\zeta_0 \in X$.
Then
\[
\phi(x)  L(\phi) X \subseteq\spa\{ u_2 x \zeta_0,   u_1x \zeta_0, u_3x \zeta_0   \}\quad(x \in A).
\]
Now the desired conclusion follows from Theorem~\ref{3-qu}.
\end{proof}
\begin{rmk}
In the above proof, if $\lDim V(\phi)=2$ and $\lDim L(\phi)=3$, then the space $\widehat{X}$
(see Lemma~\ref{mbr}) has rank at most~$2$.  Hence one can use the classification of spaces of matrices of rank at most~$2$~\cite{Atk}.
On the other hand, if $\lDim L(\phi)= 2$, one can use the structure of locally linearly dependent  spaces of dimension~$3$
which is clearly stated in \cite[Theorem~3.1]{ChSe}.  However, in both cases, the proof will be longer.
Our approach here is  direct and relatively elementary.
\end{rmk}
\begin{rmk}\label{rem:2-vs-3}
In Section~\ref{sect:locally-nil}, our standard assumption on the elementary operator $\phi\in\ElAX$
was that, for all $x\in A$ and $\zeta\in X$ satisfying $xV(\phi)\zeta\subseteq\CC\zeta$, $\phi(x)_{|L(\phi) \zeta}$ is nilpotent.
Suppose that $\phi= \sum_{i=1}^3M_{u_i,v_i}$  is an elementary operator on $A$ satisfying
\begin{equation*}
(v_iu_j)_{1 \leq i,j \leq 3}= \left(
                               \begin{array}{ccc}
                                 0 & v_2u_3 & 0 \\
                                 v_2u_1 & 0& v_2u_3 \\
                                 0 & -v_2u_1 & 0\\
                               \end{array}
                             \right).
\end{equation*}
Then $\phi$ satisfies the above condition; however, in general,
$\phi$ is not locally quasi-nilpotent. This shows that, in contrast to the case of length two, see Corollary~\ref{2-qu},
for length three these conditions are no longer equivalent.
\end{rmk}

\begin{acknowledgement}
Part of the research on this paper was carried out while the second-named author visited the Universit\' e Moulay Ismail
in Morocco. He would like to express his gratitude for the generous hospitality and support received from his colleagues there.
This visit was supported by the London Mathematical Society under their ``Research in Pairs'' programme.
\end{acknowledgement}

\smallskip


\begin{thebibliography}{18}

\bibitem{Atk}
M. D. Atkinson, \textit{Primitive spaces of matrices of bounded rank.} II,
J. Austral. Math. Soc. \textbf{34} (1983), 306--315.

\bibitem{AtLl}
M. D. Atkinson and S. Llyold,
\textit{Large spaces of matrices of bounded rank},
Quart. J. Math. \textbf{31} (1980), 253--262.

\bibitem{Aup}
B. Aupetit, \textit{A primer on spectral theory}, Springer-Verlag, New York, 1991.

\bibitem{Nad}
N. Boudi, \textit{On the product of derivations in Banach algebras},
Math. Proc. Royal Irish Acad. \textbf{109} (2009), 201--211.

\bibitem{NaMa04}
N. Boudi and M. Mathieu,  \textit{Commutators with finite spectrum},
Illinois J. Math. \textbf{48} (2004), 687--699.

\bibitem{NaMa11}
N. Boudi and M. Mathieu, \textit{Elementary operators that are spectrally bounded},
Operator Theory: Advances and Applications \textbf{212} (2011), 1--15.

\bibitem{NaMa14}
N. Boudi and M. Mathieu, \textit{More elementary operators that are spectrally bounded},
submitted.

\bibitem{BrSe}
M. Bre\v{s}ar and P. \v{S}emrl, \textit{On locally linearly dependent operators and derivations},
Trans. Amer. Math. Soc. \textbf{351} (1999), 1257--1275.

\bibitem{ChKeLee}
M. A. Chebotar, W.-F. Ke and P.-H. Lee,
\textit{On Bre\v sar--\v Semrl conjecture and derivations of Banach algebras},
Quart. J. Math. \textbf{57} (2008), 469--478.

\bibitem{ChSe}
M. A. Chebotar and P. \v{S}emrl, \textit{Minimal locally linearly dependent spaces of operators},
Linear Algebra Appl. \textbf{429} (4) (2008), 887--900.

\bibitem{CuMa}
R. E. Curto and M. Mathieu, \textit{Spectrally bounded generalized inner derivations},
Proc. Amer. Math. Soc. \textbf{123} (1995), 2431--2434.

\bibitem{Fas}
M. A. Fasoli, \textit{Classification of nilpotent linear spaces in $M(4, \CC)$},
Comm. in Algebra \textbf{25} (6) (1997), 1919--1932.

\bibitem{FoMa}
E. Fornasini and G. Marchesini, \textit{Properties of pairs of matrices and state models for two-dimensional systems},
\textit{Part~$1$: State dynamics and geometry of the pairs},
In: Multivariate Analysis: Future Directions, C. R. Rao (ed.), Elsevier Science Publishers B.V., 1993, 131--153.

\bibitem{Ger}
M. Gerstenhaber,  \textit{On nilalgebras and linear varieties of nilpotent matrices}.\!{~I},
Amer. J. Math. \textbf{80} (1958), 614--622.

\bibitem{GoLaWo}
W. Gong, D. R. Larson and W. R. Wogen, \textit{Two results on separating vectors},
Indiana Univ. Math. J. \textbf{43}  (1994), 1159--1165.

\bibitem{LiPa}
J. Li and Z. Pan, \textit{Algebraic reflexivity of linear transformations},
Proc. Amer. Math. Soc. \textbf{135 } (2007), 1695--1699.

\bibitem{MeSe}
R. Meshulam and P. \v{S}emrl, \textit{Locally linearly dependent operators},
Pacific J. Math. \textbf{203 } (2002), 441--459.

\bibitem{Pta}
V. Pt\'{a}k, \textit{Derivations, commutators and the radical},
Manuscripta Math. \textbf{23} (1978), 355--362.

\end{thebibliography}
\end{document}